\documentclass[12pt]{article}
\usepackage{graphicx}  
\usepackage{times}
\usepackage{amsmath}
\usepackage{amssymb}
\usepackage{latexsym}
\usepackage{amsthm}

\usepackage{amsfonts}

\newcommand{\hide}[1]{}

\theoremstyle{definition}
\newtheorem{theorem}{Theorem}
\newtheorem{definition}[theorem]{Definition}
\newtheorem{proposition}[theorem]{Proposition}

\newtheorem*{ores}{Ore's Theorem}
\numberwithin{theorem}{section}

\title{On the Number of Acyclic Orientations of Complete $k$-Partite Graphs}
\author{
Veselin Blagoev \\
The University of Edinburgh \\
s1136008@sms.ed.ac.uk
}
\date{}

\begin{document}
\maketitle

\begin{abstract}
\noindent
Building on previous work by Cameron et al. in \cite{cameron}, we give a recurrence for computing the number of acyclic orientations of complete $k$-partite graphs, which can be implemented to obtain a dynamic programming algorithm running in time $n^{O(k)}$, where $n$ is the number of vertices in the graph. We prove our result by using a relationship between the number of acyclic orientations and the number of Hamiltonian paths in complete $k$-partite graphs and providing a recurrence for the latter quantity. We give a simple extension of our algorithm to the situation when we are an edge removal away from having a complete $k$-partite graph.
\end{abstract}

\section{Introduction}
Let $G=(V,E)$ be a simple undirected graph with vertex set $V$ and edge set $E$. We have the following definition.

\begin{definition}\label{aoDef}
An acyclic orientation of $G$ is a (total) function
\begin{gather*}
 \sigma \colon E \to V \times V \\
 \sigma(\{u,v\}) \in \{(u,v),(v,u)\}
\end{gather*}
that assigns an orientation to each edge of $G$ such that there is no directed cycle in the resulting digraph.
\end{definition}

We denote the set of all acyclic orientations of $G$ by $\mathcal{AO}(G)$, and we are interested in computing $\vert \mathcal{AO}(G) \vert$. Richard Stanley \cite{stanley} proved the relationship between $ \mathcal{AO}(G)$ and $\chi(G;\lambda)$, the \emph{chromatic polynomial} of $G$: namely, we obtain $\vert \mathcal{AO}(G) \vert$  by evaluating $\chi(G;-1)$. Equivalently, we compute  $\vert \mathcal{AO}(G) \vert$ by evaluating the Tutte polynomial of $G$, denoted by $T(G;x,y)$, at the point $(2,0)$. There are polynomial-time algorithms for evaluating exactly $T(G;x,y)$ at every point in the plane, including $(2,0)$, for graphs $G$ of bounded treewidth (see, e.g. \cite{andrzejak}, \cite{noble}). However, for general graphs computing $T(G; 2,0)$ exactly is $\#P$-hard \cite{jaegerVertiganWelsh}. The problem remains $\#P$-complete even for planar graphs \cite{vertigan}. This has prompted work towards approximating $\vert \mathcal{AO}(G) \vert$ efficiently, with positive results for some special classes of graphs, e.g. graphs with large girth (see Bordewich \cite{bordewich}). In the context of random graphs, Reidys \cite{reidys} has shown that the number $\log(\vert \mathcal{AO}(G) \vert)$ is tightly concentrated for $G\sim \mathcal{G}_{n,p}$.

Our work follows from Cameron et al.'s \cite{cameron} on complete bipartite graphs $K_{n_1,n_2}$, where $n_1, n_2$ are the sizes of the two vertex parts. In \cite{cameron} Cameron et al. give an expression for the number of acyclic orientations of $K_{n_1,n_2}$: namely, that 
\begin{gather}\label{completeBipAOs}
\vert \mathcal{AO}(K_{n_1,n_2})\vert = \sum_{i=1}^{min\{n_1+1,n_2+1\}} (i-1)!^2S(n_1+1,i)S(n_2+1,i)
\end{gather}
where $S(a,b)$ denotes the \emph{Stirling number of the second kind} counting the number of ways to partition a set of $a$ objects into $b$ non-empty parts. 

We generalise Cameron et al.'s argument to complete $k$-partite graphs, for $k\geq 2$. While we are not able to simplify our count to a neat formula as in (\ref{completeBipAOs}), we instead give a recurrence that can be implemented in dynamic programming fashion. The rest of this note is structured as follows. In Section \ref{prelimSec} we establish our notation and give some definitions, including a particular representation of elements of $\mathcal{AO}(G)$ as employed by Reidys in \cite{reidys}, together with a bijection between the set of acyclic orientations of a complete $k$-partite graph and an appropriately defined set of permutations. In Section \ref{mainPartSec} we establish the connection between our set of permutations and the number of Hamiltonian paths of appropriately defined complete $k$-partite graphs. We present and prove our main recurrence, and we conclude with a straightforward extension to a complete $k$-partite graph with one extra edge added within a vertex part.

\section{Preliminaries}\label{prelimSec}
In the following we give a clear definition of our complete $k$-partite graphs, and we describe what turns out to be a useful representation of the acyclic orientations of a particular graph.

\begin{definition}\label{compl}
Let $k\geq 2$ be an integer, and let $\mathbf{n} \in \mathbb{Z}^k_{\geq 1}$. We denote by $G_{\mathbf{n}} = (V_{\mathbf{n}}, E_{\mathbf{n}})$ the complete $k$-partite graph where $V_{\mathbf{n}} = V_1 \cup V_2 \cup \dots\ \cup V_k$, with $\vert V_i \vert = n_i$ for all $i\in [k]$ and $V_i\cap V_j = \emptyset$ for all $i\neq j$, with $i,j\in [k]$, and $\{u,v\}\in E_{\mathbf{n}}$ if and only if $u\in V_i$ and $v\in V_j$ for some $i\neq j$, with $i,j \in [k]$.
\end{definition}

For the case where $k=2$, we get a complete bipartite graph, and Cameron et al.'s result in \cite{cameron} applies. Let $n$ denote the total number of vertices in $G_{\mathbf{n}}$: $n = \sum_{i=1}^k n_i$.

Now consider a permutation $\pi$ of $V_{\mathbf{n}}$. It induces an element $\sigma \in \mathcal{AO}(G_{\mathbf{n}})$ in the following way: for $i<j$ let $\sigma(\{\pi(i), \pi(j)\}) = (\pi(i), \pi(j))$  whenever $\{\pi(i),\pi(j)\} \in E_{\mathbf{n}}$. In other words, given a total ordering of the vertices of $G_{\mathbf{n}}$ we orient each edge from \emph{smaller} to \emph{larger} according to the ordering. Observe that $\pi$ can be viewed as a sequence of maximal blocks such that each block consists only of vertices that all come from the same vertex part $V_i$, and the blocks are maximal in the sense that we do not have two blocks with vertices from the same $V_i$ sitting next to each other in $\pi$: we would consider that as one block. As noted in \cite{reidys} and \cite{cameron}, if we are to change the order of the vertices within a single such block of $\pi$, we would still get the same $\sigma \in \mathcal{AO}(G_{\mathbf{n}})$, because there are no edges between any two vertices lying in the same block, hence, no edge would be changing its orientation. However, if we change the order of the blocks themselves, and/or we change the constitution of the blocks (i.e., which vertices go in which blocks), we obtain a different acyclic orientation.

\begin{definition}\label{partititionsOfPartitions}
For $i\in [k]$, if we have $V_i = f_{i,1} \cup \dots \cup f_{i,m_i}$ for some $m_i\leq n_i$ with $f_{i,p}\neq \emptyset$ for all $p\in [m_i]$ and $f_{i,p}\cap f_{i,q} = \emptyset$ for all $p\neq q$, with $p,q\in [m_i]$, then we call the $f_{i,p}$'s \emph{blocks} of $V_i$.
Given a collection $F = \{f_{1,1}, \dots, f_{1,m_1}, f_{2,1},\dots, f_{2,m_2},\dots, f_{k,1},\dots, f_{k,m_k}\}$ of blocks partitioning every $V_i$, let $S_F$ be the set of permutations of $F$ such that no two consecutive elements of the permutation have the same first label. Note that $S_F$ could be empty (we give a necessary and sufficient condition for $S_F$ being non-empty in Section \ref{mainSubsec}). Observe that if we take some $\pi_F\in S_F$ and \emph{unfold} the blocks into sequences of individual vertices (in any order within each block), we obtain a permutation of $V_{\mathbf{n}}$ that induces $\pi_F$ as its sequence of maximal blocks. Let $\mathcal{F}$ be the set of all such collections of blocks partitioning every $V_i$, and note that for any $F_1,F_2 \in \mathcal{F}$ with $F_1 \neq F_2$ we have $S_{F_1}\cap S_{F_2} = \emptyset$.\footnote{Even if two collections $F_1$ and $F_2$ have the same number of blocks for each $V_i$, they are still considered different collections if the blocks have different constituent vertices: i.e., if there exist two $u,v\in V_i$ for some $i\in [k]$ such that $u$ and $v$ are in the same block in $F_1$ but are in different blocks in $F_2$. Hence, the $S_{F_1}$ and $S_{F_2}$ are disjoint in the sense that if we take some $\pi_{F_1} \in S_{F_1}$ and some $\pi_{F_2} \in S_{F_2}$, unfold them as described by substituting each block with its actual constituent vertices of $G_{\mathbf{n}}$ to obtain two permutations $\pi_1$ and $\pi_2$ of $V_{\mathbf{n}}$, then the two acyclic orientations induced are different.}
\end{definition}

We let $\Phi = \bigcup_{F\in \mathcal{F}} S_F$. We prove the following observation, which is a generalisation of Cameron et al.'s argument  for complete bipartite graphs in \cite{cameron}, and is implicit in Reidy's representation of acyclic orientations in \cite{reidys}.

\begin{proposition}[\textit{Adapted from \cite{reidys}}]\label{bijectionAOPerms}
$$\vert \mathcal{AO}(G_{\mathbf{n}})\vert = \vert \Phi \vert $$
\end{proposition}
\begin{proof}
We construct a function $\psi \colon \Phi \to \mathcal{AO}(G_{\mathbf{n}})$ as follows. For $\phi \in \Phi$, let $\psi(\phi)$ be the acyclic orientation obtained by unfolding every block of $\phi$ (with individual vertices within each block ordered in any arbitrary way when unfolding, for example, lexicographically: as argued, it would not change the resulting acyclic orientation) and orienting the edges of $G_{\mathbf{n}}$ from smaller to larger as prescribed by the obtained ordering of the vertices of $G_{\mathbf{n}}$. We argue $\psi$ is a bijection. As noted in Definition \ref{partititionsOfPartitions}, $\psi$ is injective: unfolding two different permutations $\pi,\pi' \in \Phi$ gives two different acyclic orientations of $G_{\mathbf{n}}$.  Moreover, $\psi$ is surjective: we can take any $\sigma\in \mathcal{AO}(G_{\mathbf{n}})$ and we can consider any linear extension of the unique partial order defined by $\sigma$ on $V_{\mathbf{n}}$ to obtain a permutation of $V_{\mathbf{n}}$. We can then read off the blocks of that permutation: the resulting element of $\Phi$ is mapped to $\sigma$ by $\psi$. Therefore, we have that $\vert \mathcal{AO}(G_{\mathbf{n}})\vert = \vert \Phi \vert $.
\end{proof}

By Proposition \ref{bijectionAOPerms} we can compute $\vert \mathcal{AO}(G_{\mathbf{n}})\vert$ by computing $\vert \Phi \vert$. Note that we have $\vert \Phi \vert = \sum_{F\in \mathcal{F}} \vert S_F \vert$, because $S_{F_1}$ and $S_{F_2}$ are disjoint for different $F_1$ and $F_2$. We have the following definition.

\begin{definition}\label{vectorsOfPartNums}
Let $M = \{ \mathbf{m} \in \mathbb{Z}_{\geq 1}^k\ \mid \forall i \ m_i \leq n_i\}$. We will use $\mathbf{m} \in M$ as a tuple carrying the information about the number of blocks in the partition of each $V_i$ in some collection of blocks.
\end{definition}

For a collection of blocks $F$ and $\mathbf{m}\in M$ we say that $F$ agrees with $\mathbf{m}$ if for all $i\in [k]$ the number of blocks of $V_i$ in $F$ is $m_i$. Let $\mathcal{F}_{\mathbf{m}}$ be the set of collections of blocks that agree with a specific $\mathbf{m}$. We can easily give an expression for the size of $\mathcal{F}_{\mathbf{m}}$ in terms of $\mathbf{m}$.

\begin{proposition} \label{numOfPartitionsInNumParts}
For any $\mathbf{m}\in M$ we have $\vert \mathcal{F}_{\mathbf{m}} \vert = \prod\limits_{i=1}^k S(n_i, m_i)$, where $S(\cdot, \cdot)$ is the Stirling number of the second kind. 
\end{proposition} 
\begin{proof}
For two positive integer $a,b$ with $a\geq b$, the Stirling number of the second kind $S(a,b)$ is the number of partitions of an $a-$element set into $b$ parts. Therefore, $S(n_i, m_i)$ gives precisely the number of ways to partition $V_i$ into $m_i$ blocks, and we need to do that for every $V_i$.
\end{proof}

Observe that for any $\mathbf{m}\in M$ and for any $F\in \mathcal{F}_{\mathbf{m}}$ the quantity $\vert S_F \vert$ depends only on the number of blocks for each $V_i$ in $F$, but not on the actual constituent vertices of each block. Therefore, for any two collections of blocks $F_1$ and $F_2$ that agree with the same $\mathbf{m}$ we have that $\vert S_{F_1} \vert = \vert S_{F_2} \vert = s_{\mathbf{m}}$.

Combining these observations we can re-write $\vert \mathcal{AO}(G_{\mathbf{n}}) \vert$ as follows:
\begin{gather}\label{sumPhi}
\vert \mathcal{AO}(G_{\mathbf{n}}) \vert = \vert \Phi \vert = \sum_{F\in \mathcal{F}} \vert S_F \vert = \sum_{\mathbf{m} \in M} \sum_{F \in \mathcal{F}_{\mathbf{m}}} \vert S_F \vert = \sum_{\mathbf{m} \in M} s_{\mathbf{m}} \cdot \prod_{i=1}^k S(n_i, \mathbf{m}_i)
\end{gather}

In the following section we present an approach to calculating the quantities $s_{\mathbf{m}}$ for any $\mathbf{m}\in M$, thus, obtaining an algorithm for computing $\vert \mathcal{AO}(G_{\mathbf{n}}) \vert$.

\section{Algorithm for evaluating $\vert \Phi \vert$}\label{mainPartSec}

We first show an equivalence between $s_{\mathbf{m}}$ and the number of Hamiltonian paths in an appropriately defined complete $k$-partite graph. Then, we present a recurrence for computing the number of these paths. We finish the section by giving a straightforward extension to graphs that are one edge removal away from being complete $k$-partite graphs.

\subsection{Computing $s_{\mathbf{m}}$} \label{mainSubsec}
We approach the problem of computing $s_{\mathbf{m}}$ in the following way.

For $\mathbf{m}\in M$ we define the graph $G_{\mathbf{m}} = (V_{\mathbf{m}}, E_{\mathbf{m}})$ as follows: let $V_{\mathbf{m}} = \bigcup_{i=1}^k f_i$ where $\forall i \ \vert f_i\vert = m_i$ and $\forall i\neq j \ f_i \cap f_j = \emptyset$. Also, $\{a,b\}\in E_{\mathbf{m}}$ if and only if $a\in f_i$ and $b\in f_j$ for some $i\neq j$. Thus, $G_{\mathbf{m}}$ is itself a complete $k-$partite graph where the sizes of the vertex parts are given by $\mathbf{m}$. Let $HP(G_{\mathbf{m}})$ denote the set of Hamiltonian paths of $G_{\mathbf{m}}$. We prove the following proposition.

\begin{proposition}\label{HamPathsPermutes}
For any $\mathbf{m}\in M$ we have that $s_{\mathbf{m}} = \vert HP(G_{\mathbf{m}}) \vert$.
\end{proposition}
\begin{proof}
Consider any collection of blocks $F$ which agrees with $\mathbf{m}$. Then, the vertices of $G_{\mathbf{m}}$ are labelled by the blocks of $F$, and the $k$ vertex parts forming the partition of the vertex set of $G_{\mathbf{m}}$ are given by the first label of each block in $F$ (i.e., the $V_i$ that each block is originally part of). Firstly, note that any $\sigma \in S_F$ is a Hamiltonian path of $G_{\mathbf{m}}$: the $\sigma$ lists the vertices in the sequence they are visited, and by definition of an element of $S_F$ we do not have adjacent vertices from the same part, whereas we always have an edge between two vertices from different vertex parts of $G_{\mathbf{m}}$. Moreover, any $\sigma_1, \sigma_2\in S_F$ such that $\sigma_1 \neq \sigma_2$ denote two different Hamiltonian paths: since $\sigma_1 \neq \sigma_2$, there must exist a position $j$ such that $\sigma_1(j) \neq \sigma_2 (j)$, implying that the Hamiltonian path given by $\sigma_1$ differs from the Hamiltonian path given by $\sigma_2$ in the $j^{th}$ vertex visited. On the other hand, if we consider any Hamiltonian path $p$ of $G_{\mathbf{m}}$ as a sequence of vertices in the order that they are visited by $p$, clearly this sequence gives an element of $S_F$: we cannot have two vertices from the same part next to each other in $p$ since we do not have an edge within parts. Therefore, $\vert HP(G_{\mathbf{m}}) \vert = \vert S_F \vert = s_{\mathbf{m}}$.
\end{proof}

Before we proceed to prove a recurrence for the number of Hamiltonian paths in a complete $k-$partite graph $G_{\mathbf{m}}$ we need some more notation. Let $\mathbf{0}$ and $\mathbf{1}$ denote the $k-$element tuples consisting  of all zeros and of all ones, respectively. Let $\mathbf{1_i}$ denote the $k-$element tuple where the $i^{th}$ element is equal to $1$ and the remaining elements are zeros. We denote by $\mathbf{m}_{+1_i}$ the tuple where we have added $1$ to the $i^{th}$ element of $\mathbf{m}$ while keeping the other elements the same, and similarly we denote by $\mathbf{m}_{-1_i}$ the tuple where we have subtracted $1$ from the $i^{th}$ element of  $\mathbf{m}$ while keeping the other elements the same.

We give a necessary and sufficient condition for the existence of a Hamiltonian path in a complete $k-$partite graph. In order to do that, we need a classic result in graph theory.

\begin{ores}[\cite{ore}]\label{oresTh}
In any connected graph $G$ with $n\geq 3$ vertices, if for every pair of distinct non-adjacent vertices $u$ and $v$ we have that $deg_G(u) + deg_G(v) \geq n$, then $G$ has a Hamiltonian cycle.
\end{ores}

We now prove the following.

\begin{proposition}\label{oresThForHPsCompleteKPartite}
Given $\mathbf{m}\in \mathbb{Z}^k_{\geq 1}$, for all $i \in [k]$ we have $\sum_{j\neq i} m_j  \geq m_i - 1$ if and only if $s_{\mathbf{m}} > 0$.
\end{proposition}
\begin{proof}
Firstly, we show the \emph{if} direction by proving that if there exists an $i \in [k]$ such that $\sum_{j\neq i} m_j < m_i - 1$, then $s_{\mathbf{m}} = 0$. To this end, assume there exists an $i \in [k]$ with $\sum_{j\neq i} m_j < m_i - 1$. Suppose there exists a Hamiltonian path $p$ in $G_{\mathbf{m}}$. There are $\sum_{j} m_j$ vertices in $p$, of which there are at least $m_i - 1$ vertices from the $i^{th}$ part (of the vertex set) of $G_{\mathbf{m}}$ that are not the last vertex of $p$. Each of these non-last vertices has to be immediately followed by a vertex that is not in the $i^{th}$ part of $G_{\mathbf{m}}$. Therefore, we must have $\sum_{j\neq i} m_j \geq m_i - 1$, which is a contradiction. Hence, there is no Hamiltonian path, and $s_{\mathbf{m}} = 0$.

We now prove the \emph{only if} direction. Assume that for all $i \in [k]$ we have $\sum_{j\neq i} m_j  \geq m_i - 1$. Observe that if two distinct vertices are non-adjacent in $G_{\mathbf{m}}$, then they must be in the same part (of the vertex set) of $G_{\mathbf{m}}$. Recall that we denote the vertex set of $G_{\mathbf{m}}$ by $V_{\mathbf{m}} = \bigcup^k_{i = 1} f_i$ where for all $i\in k$ we have $\vert f_i \vert = m_i$. We now add a new vertex, denoted by $u$, to $G_{\mathbf{m}}$, and we add edges between $u$ and every vertex of $G_{\mathbf{m}}$ to obtain a new graph, $G'_{\mathbf{m}}$. Note that the vertex set $V'_{\mathbf{m}}$ of $G'_{\mathbf{m}}$ has size $1+ \sum_{i=1}^k m_i$. Observe that for every $i\in [k]$ we have for every distinct pair $w,z\in f_i$ the property that $deg_{G'_{\mathbf{m}}}(w) + deg_{G'_{\mathbf{m}}}(z) = 2(1 + \sum_{j\neq i} m_j)\geq m_i + 1 + \sum_{j\neq i} m_j =  \vert V'_{\mathbf{m}} \vert$, where the inequality in the middle follows from our initial assumption. Therefore, by Ore's Theorem, there exists a Hamiltonian cycle in $G'_{\mathbf{m}}$. Take any such Hamiltonian cycle and remove the vertex $u$ together with its two adjacent edges: what remains is a Hamiltonian path for $G_{\mathbf{m}}$.
\end{proof}

We can now give a recurrence for calculating the number of Hamiltonian paths in $G_{\mathbf{m}}$. Note that  $s_{\mathbf{1_i}} = 1$ for all $i\in [k]$, and $s_{\mathbf{1}} = k!$ because that is just the complete graph on $k$ vertices.

\begin{proposition}\label{recurHP}
Let $\mathbf{m} \in \mathbb{Z}_{\geq 1}^k$, and consider the tuple $\mathbf{m'} = \mathbf{m}_{+1_i}$ for some $i\in [k]$. Let $s_{\mathbf{m'}}$ denote the number of Hamiltonian paths in the complete $k-$partite graph $G_{\mathbf{m'}}$. Then, $s_{\mathbf{m'}}$ satisfies the following recurrence:
\begin{displaymath}
s_{\mathbf{m'}} =
\begin{cases}
0 & \text{if $\exists \ell\in [k]$ such that $\sum\limits_{j\neq \ell} m'_j < m'_{\ell} - 1$} \\
s_{\mathbf{m}} \cdot (1 - m_i + \sum\limits_{j\neq i} m_j) + \sum\limits_{\substack{j\neq i \\ m_j \geq 2}} m_j \cdot (m_j - 1) \cdot s_{\mathbf{m}_{-1_j}} & \text{if $\forall \ell\in [k] \ \sum\limits_{j\neq \ell} m'_j \geq m'_{\ell} - 1$}
\end{cases}
\end{displaymath}
\end{proposition}
\begin{proof}
The first case follows directly from Proposition \ref{oresThForHPsCompleteKPartite}.\newline
Consider the second case. Observe that we obtain $G_{\mathbf{m'}}$ by adding a new vertex denoted by $u$ to the vertex part $f_i$ of $G_{\mathbf{m}}$, together with edges between $u$ and every vertex not in $f_i$. By Proposition \ref{oresThForHPsCompleteKPartite} we have $s_{\mathbf{m'}} > 0$. Therefore, take a Hamiltonian path $p$ of $G_{\mathbf{m'}}$, and consider the position of $u$ in $p$. We have two main distinct cases:\newline
\textit{Case 1}: Suppose $p = \dots w u z \dots$ such that $w$ and $z$ are from different vertex parts of $G_{\mathbf{m'}}$, or $u$ is the first or last vertex of $p$. We will show how to count this type of Hamiltonian paths of $G_{\mathbf{m'}}$ when we know the number of Hamiltonian paths of $G_{\mathbf{m}}$. Now, if we let $p' = p\setminus \{u\}$ denote the path obtained by deleting $u$ together with its adjacent edges from $p$ and then reconnecting the path by inserting the edge $\{w,z\}$ (or, if $u$ is the first or last vertex of $p$, then $p'$ is already connected), then we have that $p'$ is, in fact, a Hamiltonian path for $G_{\mathbf{m}}$. Therefore, for each Hamiltonian path $p'$ of $G_{\mathbf{m}}$ we need to count the number of ways we can \emph{insert} the new vertex $u$ in $p'$, so that we obtain a Hamiltonian path for $G_{\mathbf{m'}}$. For each $p' \in HP(G_{\mathbf{m}})$ there are exactly $1 + \sum\limits_{j=1}^k m_j$ candidate positions where $u$ might be inserted. Of those, exactly $2\cdot m_i$ positions are not allowed, because we cannot insert $u$ immediately preceding or immediately following another vertex from the part $f_i$. Hence, there is a total of $1 - m_i + \sum\limits_{j\neq i} m_j$ positions where we can insert $u$ to obtain a Hamiltonian path for $G_{\mathbf{m'}}$ of the type considered here from any $p'\in HP(G_{\mathbf{m}})$. Note that for any $p'\in HP(G_{\mathbf{m}})$ and for any allowed insertion position in $p'$ we obtain a distinct $p\in HP(G_{\mathbf{m'}})$. Therefore, we have a total of $s_{\mathbf{m}} \cdot (1 - m_i + \sum\limits_{j\neq i} m_j)$ Hamiltonian paths in $G_{\mathbf{m'}}$ where the vertex $u$ is not sandwiched in between two vertices from the same vertex part.\\
\textit{Case 2}: Suppose $p = \dots w u z \dots$ such that $w,z\in f_j$ for some $j\in [k]$ (with $j\neq i$, otherwise $p$ would not constitute a Hamiltonian path). Observe that in this case $m'_j = m_j \geq 2$. Also note that $u$ is not the first or last vertex of $p$. Observe that if we remove $u$ and its two adjacent edges from $p$, we cannot reconnect the path since there is no edge between $w$ and $z$. Therefore, $p$ cannot be obtained from a Hamiltonian path in $G_{\mathbf{m}}$ as in \textit{Case 1} above. However, observe that we can construct a path $p'$ by removing $u,z$ and the edges adjacent to them from $p$, and then reconnecting the path by inserting the edge between $w$ and the vertex following $z$ (if any: $z$ might be the last vertex of $p$, in which case $p'$ is already connected), which must exist: since $w$ and $z$ are in the same vertex part of $G_{\mathbf{m'}}$, then the vertex immediately following $z$ (if any) is in a different vertex part from $w$. The path $p'$ obtained this way is a Hamiltonian path for the graph $G_{\mathbf{m}_{-1_j}}$, which can be viewed as the graph obtained from $G_{\mathbf{m}}$ by deleting the vertex $z$ together with all its adjacent edges from $G_{\mathbf{m}}$. Therefore, for each possible choice of $z$, of which there are $m_j$, we construct $G_{\mathbf{m}_{-1_j}}$ by removing $z$ and its adjacent edges from $G_{\mathbf{m}}$, and then we can take any $p' \in HP(G_{\mathbf{m}_{-1_j}})$, and we can insert $u$ and $z$ together (and the three, or two, if inserting at the end of $p'$, adjacent edges) in any of the $m_j - 1$ positions in $p'$ immediately following a vertex from $f_j$. Therefore, for each particular choice of $z$ we have a total of $(m_j - 1) \cdot s_{\mathbf{m}_{-1_j}}$ Hamiltonian paths of $G_{\mathbf{m'}}$ obtained that way. Hence, we have a total of $m_j \cdot (m_j - 1) \cdot s_{\mathbf{m}_{-1_j}}$ Hamiltonian paths of the type considered here for each choice of $j\neq i$ such that $m_j \geq 2$, and then we sum over all such $j$.
\end{proof}

The recurrence in Proposition \ref{recurHP} suggests a dynamic programming algorithm for computing $s_{\mathbf{m}}$ for any $\mathbf{m} \in M$. We can pre-compute all the required $\vert M \vert = \prod_{i=1}^k n_i \leq n^k$ values in a $k$-dimensional array at the start in a dynamic programming fashion, so that when we calculate each $s_{\mathbf{m}}$ we have already computed the quantities required in the recurrence. Moreover, we can pre-compute all the necessary Stirling numbers of the second kind in a dynamic programming fashion by using the recurrence $S(a,b) = b \cdot S(a-1,b) + S(a-1,b-1)$. Then, we evaluate the sum in (\ref{sumPhi}) by looking up each of the required terms in our pre-computed arrays. Overall, in this way we compute $\vert \mathcal{AO}(G_{\mathbf{n}}) \vert$ in at most $n^{O(k)}$ steps. Thus, when $k$ is constant, we obtain an algorithm for $\vert \mathcal{AO}(G_{\mathbf{n}}) \vert$ that runs in time polynomial in the number of vertices in the graph.

\subsection{Edge removal away from a complete $k$-partite graph}\label{extensionSubsec}
Suppose we have a graph $G'$ obtained from a complete $k$-partite graph $G_{\mathbf{n}}$ by adding an edge between two vertices in the $i^{th}$ vertex part in $G_{\mathbf{n}}$. Let that edge be $\{u,v\}$. Then, using the original deletion-contraction property from \cite{stanley} we have that

$$\vert \mathcal{AO}(G')\vert = \vert \mathcal{AO}(G' \setminus \{u,v\}) \vert + \vert AO(G'/\{u,v\}) \vert $$

Of course, by construction we have that $\vert \mathcal{AO}(G' \setminus \{u,v\}) \vert = \vert \mathcal{AO}(G_{\mathbf{n}}) \vert$. Moreover, since the vertices $u$ and $v$ have exactly the same neighbourhood - namely, the vertices $\bigcup_{j\neq i} V_j$ together with $u$ and $v$ themselves - it follows that the graph $G'/\{u,v\}$ is a complete $k-$partite graph with almost the same vertex set as $G_{\mathbf{n}}$ except that the $i^{th}$ part of $G'/\{u,v\}$ has one less vertex than the $i^{th}$ part of  $G_{\mathbf{n}}$. Therefore, we can compute $\vert \mathcal{AO}(G')\vert$ in time $n^{O(k)}$ by using two calls to the algorithm described in Section \ref{mainSubsec} above.

\section{Conclusion}
We have given a polynomial time algorithm for computing the number of acyclic orientations of a complete $k$-partite graph when $k$ is constant. The algorithm we have presented is based on a recurrence derived from reduction to counting the number of Hamiltonian paths in complete $k$-partite graphs. Our result extends previous work by Cameron et al. \cite{cameron}, although we are unable to give a similarly neat formula when $k>2$. A natural future extension is to consider what happens if we remove an edge from a complete $k$-partite graph, and how to adapt our algorithm to that resulting graph.

\section*{Acknowledgements}

This work has benefited from conversations with and comments on initial draft by Mary Cryan.

\end{document}